\documentclass{amsart}

\usepackage{indentfirst}
\usepackage{amssymb}
\usepackage{amsmath}
\usepackage{graphics}
\usepackage{epsfig}
\usepackage[usenames]{color}
\usepackage{MnSymbol,wasysym}

\usepackage[utf8]{inputenc}
\usepackage{latexsym}
\usepackage{graphicx}

\usepackage[normalem]{ulem}

\usepackage{hyperref}

\usepackage[T1]{fontenc}
\usepackage{kantlipsum}
\usepackage[usenames,dvipsnames]{xcolor}
\usepackage[breakable, theorems, skins]{tcolorbox}
\tcbset{enhanced}

\definecolor{light-gray}{gray}{0.95}

\catcode`@=11 \@addtoreset{equation}{section}
\renewcommand\theequation{\thesection.\@arabic\c@equation}
\catcode`@=12

\newcommand{\RR}{\mathbb{R}}

\expandafter\chardef\csname pre amssym.def
at\endcsname=\the\catcode`\@ \catcode`\@=11
\def\undefine#1{\let#1\undefined}
\def\newsymbol#1#2#3#4#5{\let\next@\relax
 \ifnum#2=\@ne\let\next@\msafam@\else
 \ifnum#2=\tw@\let\next@\msbfam@\fi\fi
 \mathchardef#1="#3\next@#4#5}
\def\mathhexbox@#1#2#3{\relax
 \ifmmode\mathpalette{}{\m@th\mathchar"#1#2#3}%
 \else\leavevmode\hbox{$\m@th\mathchar"#1#2#3$}\fi}
\def\hexnumber@#1{\ifcase#1 0\or 1\or 2\or 3\or 4\or 5\or 6\or 7\or 8\or
 9\or A\or B\or C\or D\or E\or F\fi}

\font\teneufm=eufm10 \font\seveneufm=eufm7 \font\fiveeufm=eufm5
\newfam\eufmfam
\textfont\eufmfam=\teneufm \scriptfont\eufmfam=\seveneufm
\scriptscriptfont\eufmfam=\fiveeufm

\catcode`\@=\csname pre amssym.def at\endcsname

\newcommand{\eqn}{\begin{eqnarray}}
\newcommand{\een}{\end{eqnarray}}

\newtheorem {Theorem}  {Theorem}

\numberwithin{Theorem}{section}
\newtheorem{Lemma}[Theorem]{Lemma}
\newtheorem{Proposition}[Theorem]{Proposition}
\newtheorem{Definition}[Theorem]{Definition}
\newtheorem{Remark}[Theorem]{Remark}

\renewcommand{\a}{\alpha}

\renewcommand{\div}{\mbox{div}}

\begin{document}

\title[Decay of the 4D critical nonlinear heat equation]{DECAY RATES FOR THE 4D ENERGY-CRITICAL NONLINEAR HEAT EQUATION
}

\author[L. Kosloff]{Leonardo Kosloff}
\address[L. Kosloff]{Departamento de Matem\'atica, Instituto de Matem\'atica, Estat\'{\i}stica e
Computa\c{c}\~ao Cient\'{\i}fica, Universidade Estadual de Campinas, Rua Sergio Buarque de Holanda, 651, 13083-859, Campinas - SP, Brazil}
\email{kosloff@ime.unicamp.br}

\author[C. J. Niche]{C\'esar J. Niche}
\address[C.J. Niche]{Departamento de Matem\'atica Aplicada, Instituto de Matem\'atica. Universidade Federal do Rio de Janeiro, CEP 21941-909, Rio de Janeiro - RJ, Brazil}
\email{cniche@im.ufrj.br}

\author[G. Planas]{Gabriela Planas}
\address[G. Planas]{Departamento de Matem\'atica, Instituto de Matem\'atica, Estat\'{\i}stica e
Computa\c{c}\~ao Cient\'{\i}fica, Universidade Estadual de Campinas, Rua Sergio Buarque de Holanda, 651, 13083-859, Campinas - SP, Brazil}
\email{gplanas@unicamp.br}

\thanks{L. Kosloff has been supported by FAPESP-Brazil  grant 2016/15985-0.  C.J. Niche acknowledges support from PROEX - CAPES.  G. Planas was partially supported by CNPq-Brazil, grant 310274/2021-4, and FAPESP-Brazil grant 2019/02512-5}



\date{\today}

\begin{abstract}
In this paper we address the decay of solutions to the four-dimen\-sional energy-critical nonlinear heat equation in the critical space $\dot{H}^1$. Recently,  it was proven that the  $\dot{H}^1$ norm of solutions goes to zero when time goes to infinity,  but no decay rates were established.  By means of the Fourier Splitting Method and using properties arising from the scale invariance,  we obtain an algebraic upper bound for the decay rate of solutions.
\end{abstract}

\maketitle

\section{Introduction}

The energy-critical nonlinear heat equation in $\RR^n$,
\begin{align}
\label{eq:critical-nonlinear-heat}
\partial _t u   & =  \Delta u +|u|^{\frac{4}{n-2}} u,   \nonumber \\  u(x,0) & = u_0 (x) \in \dot{H}^1 (\RR ^n),
\end{align}
where $n \geq 3$, has recently been extensively studied due to its rich structure, which naturally leads to important questions.  Using the natural scaling
\begin{equation}
\label{eq:scaling}
u _{\lambda} (x,t) = \lambda ^{\frac{n-2}{2}} u (\lambda x, \lambda^2 t), \quad \lambda > 0,
\end{equation}
from a solution $u$ we obtain a new one $u_{\lambda}$  such that in the critical space $\dot{H}^1$ we have
\begin{displaymath}
\Vert u_{\lambda} (t) \Vert _{\dot{H}^1} = \Vert u (t) \Vert _{\dot{H}^1}, \qquad t > 0.
\end{displaymath}
This scaling leaves the energy
\begin{equation}
\label{eq:energy-general-case}
E( u(t)) = \frac{1}{2} \int _{\RR^n}  |\nabla u (t) |^2  \, dx - \frac{n-2}{2n} \int _{\RR^n}  |u (t)|^{\frac{2n}{n-2}} \, dx,
\end{equation}
invariant, i.e. $E (u_{\lambda} (t)) = E (u(t))$, for all $t, \lambda > 0$.  Moreover,  each of the terms in \eqref{eq:energy-general-case} is also invariant by this scaling.  Equation \eqref{eq:critical-nonlinear-heat} has a stationary solution  with remarkable properties,  the Aubin-Talenti bubble \cite{MR431287}, \cite{MR463908}
\begin{displaymath}
W(x) = \frac{1}{\left( 1 + \frac{1}{n(n-2)}  |x|^2 \right) ^{\frac{n-2}{2}}},
\end{displaymath}
which is,  up to scaling by \eqref{eq:scaling} and translations,  the unique  extremum of the Sobolev embedding $\dot{H} ^1 (\RR ^n) \subset L^{\frac{2n}{n-2}} (\RR^n)$,  i.e. $W$ realizes the equality in

\begin{displaymath}
\frac{\Vert u \Vert _{L ^{\frac{2n}{n-2}}}}{ \Vert u \Vert _{\dot{H}^1}} \leq C, \quad {\text{for} \, \, \, } C = \sqrt{\frac{1}{\pi n (n-2)}}  \left(\frac{\Gamma(n)}{\Gamma \left( \frac{n}{2} \right)} \right)^{\frac{1}{n}}\,.
\end{displaymath}

\begin{Remark} Equation \eqref{eq:critical-nonlinear-heat} is a particular case of the semilinear heat equation

\begin{equation}
\label{eq:fujita}
\partial _t u   =  \Delta u +|u|^{p-1} u, \qquad p > 1
\end{equation}
which has been extensively studied since Fujita's pioneering work \cite{MR214914}.  For a detailed account of results on this equation,  see Quittner and Souplet \cite{MR3967048}.
\end{Remark}

As usual for equations of the form \eqref{eq:fujita},   it is very important to establish whether solutions are global in time or their norm grows to infinity in a finite time., i.e.  they have finite time blow-up.   This problem has been studied for \eqref{eq:critical-nonlinear-heat}  by  del Pino,  Musso and Wei \cite{MR3952701},  del Pino,  Musso, Wei and Zhou \cite{MR4097503},  Filippas, Herrero and Vel\'azquez \cite{MR1843848},  Galaktionov and King \cite{MR1968320},  Harada \cite{MR4072807}, \cite{MR4151635},  Schweyer \cite{MR2990063}.  The existence of solutions which exhibit infinite time blow-up,  this is,  global solutions such that  their norm goes to infinity when time goes to infinity,  was conjectured by Fila and King \cite{MR3004680} and recently established by Cort\'azar, del Pino and Musso \cite{MR4046015},  del Pino, Musso and Wei \cite{MR4047646},  \cite{MR4307216}.   Collot,  Merle and Raphael \cite{MR3623259} proved  that for $n \geq 7$ it is possible to classify the behaviour of solutions with initial data close to the ground state $W$ and that this rules out certain kinds of blow-up. There are also many important results for solutions to \eqref{eq:critical-nonlinear-heat} in bounded domains,  for details we refer to del Pino,  Musso and Wei \cite{MR4157676},  del Pino,  Musso,  Wei and Zheng \cite{MR4288612},  Ikeda and Taniguchi \cite{https://doi.org/10.48550/arxiv.1902.01016},  Mizoguchi and Soupplet \cite{MR3994444},  Suzuki \cite{MR2492236},  Wang and Wei \cite{https://doi.org/10.48550/arxiv.2101.07186}.

In this article we are interested in the $4D$ version of \eqref{eq:critical-nonlinear-heat}, this is

\begin{align}
\label{eq:critical-nonlinear-heat-4d}
\partial _t u   & =  \Delta u +|u|^2 u,   \\  u(x,0) & = u_0 (x) \in \dot{H} ^1 (\RR^4). \nonumber
\end{align}
In this case,  the Aubin-Talenti bubble becomes

\begin{equation}
\label{eq:bubble-solution}
W(x) = \frac{1}{1 + \frac{1}{8} |x|^2}\,,
\end{equation}
and the Sobolev constant is $C = \Vert \nabla W \Vert _{L^2} ^{- \frac{1}{2}}$, see Gustafson and Roxanas \cite{MR3765769}. We note that $W$ is in $\dot{H} ^1 (\RR^4)$ but not in $L^2(\RR^4)$, and the energy is
 \[
E( u(t)) = \frac{1}{2} \int _{\RR^4}  |\nabla u (t) |^2  \, dx - \frac{1}{4} \int _{\RR^4}  |u (t)|^{4} \, dx.
\]

Our main goal is to study the decay of the $\dot{H} ^1$ norm of solutions to \eqref{eq:critical-nonlinear-heat-4d}. Recently Gustafson and Roxanas \cite{MR3765769} proved global existence of mild solutions to \eqref{eq:critical-nonlinear-heat-4d} provided the initial datum $u_0$ is small compared to the ground state \eqref{eq:bubble-solution} and established its behaviour at infinity.

\begin{Theorem}[Theorem 1.1, Gustafson and Roxanas \cite{MR3765769}]
\label{theorem-gustafson-roxanas}
Let $u_0 \in \dot{H} ^1 (\RR^4)$ such that
\begin{equation}
\label{eqn:conditions-gustafson-roxanas}
E(u_0) \leq E(W), \quad \Vert \nabla u_0 \Vert _{L^2} \leq \Vert \nabla W \Vert _{L^2}.
\end{equation}
Then the mild solution to \eqref{eq:critical-nonlinear-heat} is global and
\begin{equation}
\label{eq:decay-gustafson-roxanas}
\lim _{t \to \infty} \Vert u(t) \Vert _{\dot{H} ^1} = 0.
\end{equation}

\end{Theorem}
The  proof of \eqref{eq:decay-gustafson-roxanas} follows along the lines of that in Gallagher, Iftimie and Planchon \cite{MR1891005},  who proved an analogous result for the critical $\dot{H}^{\frac{1}{2}}$ norm of solutions to the $3D$ Navier-Stokes equations.  However,  the method's nature does not allow to infer  decay rates in \eqref{eq:decay-gustafson-roxanas}.

We now briefly describe the elements we need to state  our main result,  which says that the decay in \eqref{eq:decay-gustafson-roxanas} is in fact algebraic,  with rate depending on the decay character of the initial datum.   The decay character of $v_0 \in L^2 (\RR^n)$,  denoted by $r^{\ast} = r^{\ast} (v_0)$,  is a number $- \frac{n}{2} < r^{\ast} < \infty$ which roughly speaking shows that $|\widehat{v_0} (\xi)| \approx |\xi| ^{r^{\ast}}$, when $|\xi| \approx 0$.  If  $v$ is the solution to the heat equation with initial datum $v_0 \in L^2(\RR^n)$, then
\begin{displaymath}
C_1 (1 + t)^{- \left( \frac{n}{2} + r^{\ast} \right)} \leq \Vert v(t) \Vert _{L^2} ^2 \leq C_2 (1 + t)^{- \left( \frac{n}{2} + r^{\ast} \right)}, \qquad C_1, C_2 > 0.
\end{displaymath}
The decay character was introduced by Bjorland and M.E. Schonbek \cite{MR2493562},  and later refined and extended for a large family of dissipative linear operators by Niche and M.E. Schonbek \cite{MR3355116}  and Brandolese \cite{MR3493117}, see Section \ref{settings} for further details.

We can now state our main result.  Let $\Lambda = (- \Delta) ^{\frac{1}{2}}$.

\begin{Theorem}
\label{main-theorem}
Let $u_0 \in \dot{H} ^1 (\RR^4)$ with $q^{\ast} = r^{\ast} \left( \Lambda u_0\right) > -2$, and $\Vert u_0 \Vert _{\dot{H}^1}$ small enough.  Then for the global solution $u$ to \eqref{eq:critical-nonlinear-heat} we have

\begin{equation}
\label{eq:main-estimate}
\Vert u (t) \Vert ^2 _{\dot{H} ^1} \leq C (1+t) ^{- \min \left\{ \left(2 + q^{\ast}\right), 1 \right\}}.
\end{equation}
\end{Theorem}

In order to prove Theorem \ref{main-theorem} we follow the idea used by the authors in \cite{https://doi.org/10.48550/arxiv.2206.09445}  to show that the critical $\dot{H}^{\frac{1}{2}}$ norm of solutions to the $3D$ Navier-Stokes equations has algebraic decay rates depending on the decay character of the initial datum.  More precisely,  in Proposition \ref{lyapunov-function}  we prove that the $\dot{H} ^1$ norm is a Lyapunov function for solutions to  \eqref{eq:critical-nonlinear-heat}, where we use in an essential way the fact that,  as solutions are obtained through a fixed point argument,  $\Vert u(t) \Vert^2 _{\dot{H}^1}$ can be made uniformly small in time.   Through Proposition \ref{lyapunov-function} we rigurously obtain the  inequality

\[
\frac{d}{dt} \Vert u(t) \Vert^2 _{\dot{H}^1} \leq - 2 \left( 1- C  \Vert u(t) \Vert^2 _{\dot{H}^1} \right) \Vert \nabla u(t) \Vert^2 _{\dot{H}^1}\,,
\]
which leads,  after once again using that $\Vert u(t) \Vert^2 _{\dot{H}^1}$ is small,  to

\begin{displaymath}
\frac{d}{dt} \Vert u(t) \Vert^2 _{\dot{H}^1} \leq - \widetilde{C} \Vert \nabla u(t) \Vert^2 _{\dot{H}^1}.
\end{displaymath}
This  is the starting point for using the Fourier Splitting method,  developed by M.E.  Schonbek \cite{MR571048},  \cite{MR775190}, \cite {MR837929} to prove decay of solutions to parabolic conservation laws and the Navier-Stokes equations.  A careful adaptation of the method to our context leads to the estimate.  Our proof is independent of the decay result in Theorem \ref{theorem-gustafson-roxanas},  as we do not use at any time the fact  that  $\Vert u(t) \Vert^2 _{\dot{H}^1}$ goes to zero.  However,   both Theorems \ref{theorem-gustafson-roxanas} and \ref{main-theorem}  crucially depend on the small size of $u_0 \in \dot{H} ^1 (\RR^4)$ beyond the existence of solutions result.  More precisely,  in our case a small $u_0$ leads to a uniformly in time small  solution $u(t)$, which is key to our proof,  as we already pointed out.  In Gustafson and Roxanas \cite{MR3765769},  if

\begin{displaymath}
\Vert \nabla u_0 \Vert _{L^2} < \frac{1}{\sqrt{2}} \Vert \nabla W \Vert _{L^2},
\end{displaymath}
then the inequality $E(u_0) \leq E(W)$ follows after a simple calculation noting that $ E(W) =   \frac{1}{4} \Vert W \Vert  _{\dot{H}^1} ^2$,  where we used that the Sobolev constant is $C = \Vert \nabla W \Vert _{L^2} ^{- \frac{1}{2}}$.  Hence,  a small $u_0$ leads to both conditions in \eqref{eqn:conditions-gustafson-roxanas} being fulfilled simultaneously.

The decay character can be used in  \eqref{eq:main-estimate} to establish the role of the linear and nonlinear part of the solution in the decay: when $q^{\ast} < -1$,  then $\min \left\{ \left(2 + q^{\ast}\right), 1 \right\} = 2 + q^{\ast}$,  hence the decay is driven by the slower linear part,  see \eqref{eq:decay-linear-part}.  When $q^{\ast} > -1$,  the slower decay is that of the nonlinear part.

\begin{Remark}
In their Remark 1.1,  Gustafson and Roxanas \cite{MR3765769} note that they expect \eqref{eq:decay-gustafson-roxanas} to hold for any $n \geq 3$.  We believe that a version of Theorem \ref{main-theorem} should hold for $n \geq 3$.  We will address this generalization elsewhere,  as we would have to carefully set the space in which the existence results hold,  see Definition \ref{definition-solution}.  This is key for proving Proposition \ref{lyapunov-function} and \eqref{eq:preparation-fourier-splitting}, through \eqref{eq:condition-existence} and Theorem   \ref{gustafson-roxanas}.
\end{Remark}

This paper is organized as follows.  In Section \ref{settings}  we recall the definition and properties of the Decay Character and in Section \ref{proofs} we prove our main result.

\section{Decay Character}

\label{settings}

In order to express the decay rate of solutions to dissipative equations,  Bjorland and M.E. Schonbek \cite{MR2493562} introduced the decay character, which was later refined and extended by Niche and M.E. Schonbek \cite{MR3355116}, and Brandolese \cite{MR3493117}.  Roughly speaking,  to an initial datum  $v_0 \in L^2(\RR^n)$ the decay character associates a number which measures its ``algebraic order'' near the origin, by comparing $ |\widehat{v_0} (\xi)|$ to $f(\xi) = |\xi|^{r}$ at $\xi = 0$.

\begin{Definition} \label{decay-indicator}
Let  $v_0 \in L^2(\RR^n)$. For $r \in \left(- \frac{n}{2}, \infty \right)$, we define the {\em decay indicator}  $P_r (v_0)$ corresponding to $v_0$ as
\begin{displaymath}
P_r(v_0) = \lim _{\rho \to 0} \rho ^{-2r-n} \int _{B(\rho)} \bigl |\widehat{v_0} (\xi) \bigr|^2 \, d \xi,
\end{displaymath}
provided this limit exists. In the expression above,  $B(\rho)$ denotes the ball at the origin with radius $\rho$.
\end{Definition}

\begin{Definition} \label{df-decay-character} The {\em decay character of $ v_0$}, denoted by $r^{\ast} = r^{\ast}( v_0)$ is the unique  $r \in \left( -\frac{n}{2}, \infty \right)$ such that $0 < P_r (v_0) < \infty$, provided that this number exists. We set $r^{\ast} = - \frac{n}{2}$, when $P_r (v_0)  = \infty$ for all $r \in \left( - \frac{n}{2}, \infty \right)$  or $r^{\ast} = \infty$, if $P_r (v_0)  = 0$ for all $r \in \left( -\frac{n}{2}, \infty \right)$.
\end{Definition}

It is possible to explicitly compute the decay character for many important examples. When $v_0 \in L^p (\RR^n) \cap L^2 (\RR ^n)$ for $1 < p < 2$ and $v_0 \notin L^{\bar{p}} (\RR ^n)$ for $\bar{p} < p$, we have that $r^{\ast} (v_0) = - n \left( 1 - \frac{1}{p} \right)$, see Example 2.6 in Ferreira, Niche and Planas \cite{MR3565380}. When $v_0 \in L^1 (\RR^n) \cap L^2 (\RR^n)$ and $|\widehat{v_0} (\xi)|$ is bounded away from zero, i.e.
\begin{displaymath}
C_1 \leq |\widehat{v_0} (\xi)| \leq C_2, \qquad |\xi| \leq \beta,
\end{displaymath}
for some $\beta > 0$ and $0 < C_1 \leq C_2$, then $r^{\ast} (v_0) = 0$, see Section 4.1 in M.E. Schonbek \cite{MR837929}. For  $v_0 \in L^{1, \gamma} (\RR^n)$, for $0 \leq \gamma \leq 1$, i.e.
\begin{displaymath}
%\label{eqn:l-uno-gamma}
\Vert f \Vert _{L^{1, \gamma} (\mathbb{R} ^n)} = \int _{\mathbb{R} ^n} \left( 1 + |x| \right) ^{\gamma} |f(x)| \, dx < \infty,
\end{displaymath}
we have from Lemma 3.1 in Ikehata \cite{MR2055280} that

\begin{displaymath}
|\widehat{v_0} (\xi) | \leq C |\xi| ^{\gamma} + \left| \int _{\RR ^n} v_0(x) \, dx \right|, \qquad \xi \in \RR ^n.
\end{displaymath}
If $v_0$ has zero mean,  then $r^{\ast} (v_0) = \gamma$, if not, then $r^{\ast} (v_0) = 0$.

Bjorland and M.E. Schonbek \cite{MR2493562} used the decay character to obtain upper and lower decay rates for solutions to the heat equation.  We describe now a more general setting, due to Niche and M.E. Schonbek \cite{MR3355116},  in which the decay character provides decay estimates.  Let $X$ be a Hilbert space and consider a linear diagonalizable pseudodifferential operator $\mathcal{L}: X^n \to \left( L^2 (\RR^n) \right) ^n$,  with symbol matrix $ M(\xi)$ such that
\begin{equation}
\label{eqn:symbol}
M(\xi) = P^{-1} (\xi) D(\xi) P(\xi), \qquad \xi-a.e.
\end{equation}
where $P(\xi) \in O(n)$ and $D(\xi) = - c_i |\xi|^{2\a} \delta _{ij}$, for $c_i > c>0$ and $0 < \a \leq 1$. Then,  from the linear system $ v_t = \mathcal{L} v$ we obtain
\begin{displaymath}
\frac{1}{2} \frac{d}{dt} \Vert v(t) \Vert _{L^2} ^2  \leq  - C  \int _{\RR^n} |\xi|^{2 \a} |\widehat{v}|^2 \, d \xi,
\end{displaymath}
which is the starting point for using the Fourier Splitting Method,  see \eqref{eq:preparation-fourier-splitting} - \eqref{eqn:key-inequality} in Section \ref{proofs} for an example on how this method is used.  We note that the vectorial fractional Laplacian $(- \Delta) ^{\alpha}$, with $0 < \alpha \leq 1$; the vectorial Lam\'e  operator
$\mathcal{L} u = \Delta u + \nabla \, \div \, u, $
see  Example 2.9 in Niche and M.E. Schonbek \cite{MR3355116}, and the linear part of the magneto-micropolar system (see Niche and Perusato \cite{MR4379088}), amongst others, are diagonalizable as in \eqref{eqn:symbol}.

We now state the Theorem that describes decay in terms of the decay character for linear operators as in \eqref{eqn:symbol}.

\begin{Theorem}{(Theorem 2.10, Niche and M.E. Schonbek \cite{MR3355116})}
\label{characterization-decay-l2}
Let $v_0 \in L^2 (\RR^n)$ have decay character $r^{\ast} (v_0) = r^{\ast}$. Let $v (t)$ be the solution to the linear system

\begin{displaymath}
v_t = \mathcal{L} v
\end{displaymath}
with initial datum $v_0$, where the operator $\mathcal{L}$ is such that \eqref{eqn:symbol}  holds. Then if $- \frac{n}{2 } < r^{\ast}< \infty$, there exist constants $C_1, C_2> 0$ such that
\begin{displaymath}
C_1 (1 + t)^{- \frac{1}{\a} \left( \frac{n}{2} + r^{\ast} \right)} \leq \Vert v(t) \Vert _{L^2} ^2 \leq C_2 (1 + t)^{- \frac{1}{\a} \left( \frac{n}{2} + r^{\ast} \right)}.
\end{displaymath}
\end{Theorem}

 When defining the decay character in Definitions  \ref{decay-indicator} and \ref{df-decay-character},  it is assumed that the limits in those expressions exist.  Brandolese \cite{MR3493117} showed that this is not necessarily the case for all $v_0 \in L^2 (\RR ^n)$,  by constructing highly oscillating near the origin initial data, for which the limit in Definition \ref{decay-indicator} does not exist for some $r$.  However,  he gave a slightly different definition of decay character, more general than that in  Definitions \ref{decay-indicator} and \ref{df-decay-character},  which produces the same result when these hold.  Moreover,  he showed that the decay character $r ^{\ast}$ exists for $v_0 \in L^2 (\RR ^n)$ if and only if $v_0$ belongs to a specific subset of a homogeneous Besov space, i.e. $v_0  \in  \dot{\mathcal{A}} ^{- \left(\frac{n}{2} + r^{\ast} \right)} _{2, \infty} \subset \dot{B} ^{- \left(\frac{n}{2} + r^{\ast} \right)} _{2, \infty}$ and also that, for diagonalizable linear operators $\mathcal{L}$ as in \eqref{eqn:symbol},  the estimates
\begin{displaymath}
C_1 (1 + t)^{- \frac{1}{\a} \left( \frac{n}{2} + r^{\ast} \right)} \leq \Vert v(t) \Vert _{L^2} ^2 \leq C_2 (1 + t)^{- \frac{1}{\a} \left( \frac{n}{2} + r^{\ast} \right)},
\end{displaymath}
hold if and only if the decay character $r^{\ast} = r^{\ast} (v_0)$ exists. This provides a complete and sharp characterization of algebraic decay rates for such systems and furnishes a key tool for studying decay for nonlinear systems.

\section{Proof of Theorem \ref{main-theorem}}
\label{proofs}

\subsection{Global mild solutions for \eqref{eq:critical-nonlinear-heat} }
We recall now  definitions and results from Gustafson and Roxanas \cite{MR3765769} concerning existence of global solutions to  \eqref{eq:critical-nonlinear-heat}.

\begin{Definition} \label{definition-solution} We say $u$ is a global mild solution to \eqref{eq:critical-nonlinear-heat} if
\[
u(t) = e^{t \Delta} u_0 + \int_0^t e^{(t-s) \Delta} |u|^2 u \,ds
\]
and $u \in C^0 ([0,\infty); \dot{H} ^1 (\RR^4)) \cap L^6 ([0,\infty); L^6 (\RR^4))$, $\nabla u \in L^3 ([0,\infty); L^3(\RR^4))$, $D^2 u \in L^2([0,\infty); L^2(\RR^4))$, $\partial_t u \in L^2([0,\infty); L^2(\RR^4))$.
\end{Definition}
Now consider
\begin{displaymath}
 \| u \|_{L_t^6 L_x^6}\, = \left( \int_{\RR _{+}} \int_{\RR^4} |u(t,x)|^6 \right)^{1/6}.
\end{displaymath}
Then,  from well-known heat equation estimates (see Giga \cite{MR833416}) and the Sobolev embedding $\dot{H} ^1 (\RR^4) \subset L^4 (\RR^4)$, we have for $u_0 \in \dot{H} ^1(\RR^4)$
\begin{equation}
\label{eq:condition-existence}
\| e^{t \Delta} u_0 \|_{L_t^6 L_x^6} \leq C \| u_0 \|_{L^4}\,\leq C \| u_0 \|_{\dot{H}^1}\,,\; t>0\,.
\end{equation}

We recall the small data global existence result.

\begin{Theorem}[Theorem 2.1 (5),  Gustafson and Roxanas \cite{MR3765769}]
\label{gustafson-roxanas}
There is an $\epsilon_0 > 0$ such that if

\begin{displaymath}
\| e^{t \Delta} u_0 \|_{L_t^6 L_x^6} \leq \epsilon_0,
\end{displaymath}
then \eqref{eq:critical-nonlinear-heat} has a global solution and

\begin{equation}
\label{eqn:small-solution}
 \| u \|_{L_t^6 L_x^6} + \Vert \nabla u \Vert _{L^{\infty} _t L^2 _x \cap L^3 _t L^3 _ x} + \Vert D^2 u \Vert _{L^2 _t L^2 _x } \leq C \epsilon_0, \qquad C > 0.
\end{equation}
Note that,  by \eqref{eq:condition-existence},  this holds if $ \| u_0 \|_{\dot{H}^1}$ is small enough.
\end{Theorem}

\subsection{Proof of Theorem \ref{main-theorem}} In order to prove that the $\dot{H}^1$ norm is a Lyapunov function,  first key step in the proof of Theorem \ref{main-theorem}, we need the following Proposition.

\begin{Lemma} [Lemma 5.10, Bahouri, Chemin and Danchin \cite{MR2768550}] \label{Lemma-heat}
Let $w \in C\left([0,T]; \mathcal{S} '(\RR ^n) \right)$ be the solution to

\begin{align*}
\partial _t w - \Delta w  & =   f,  \nonumber \\ w(x,0) & = w_0 (x) ,
\end{align*}
where $f \in L^2 \left([0,T]; \dot{H} ^{s-1}(\RR^n) \right)$ and $w_0 \in \dot{H}^{s} (\RR^n)$.  Then we have the energy identity

\[
\Vert w(t) \Vert ^2 _{\dot{H} ^s}  + 2 \int _0 ^t \Vert \nabla w(\tau) \Vert ^2  _{\dot{H}^s} \, d\tau = \Vert w_0 \Vert ^2 _{\dot{H} ^s} + 2  \int _0 ^t \langle w (\tau), f(\tau) \rangle _{\dot{H}^s} \, d \tau.
\]
\end{Lemma}

We can now prove such result.

\begin{Proposition}
\label{lyapunov-function}
Let $u_0 \in \dot{H}^{1} (\RR ^4)$  be small enough. Then,
the norm in $\dot{H}^{1} (\RR ^4)$ is a Lyapunov function for equation \eqref{eq:critical-nonlinear-heat}.
\end{Proposition}

\begin{proof}
Let $f = |u|^2 u$,  hence from  Definition \ref{definition-solution} we have that $f \in  L^2([0,\infty); L^2(\RR^4))$.  Taking $s = 1$,  $w_0 = u_0 \in \dot{H}^{1}(\RR^4)$ and applying Lemma \ref{Lemma-heat} we obtain the energy identity

\[
\Vert u(t_2) \Vert ^2 _{\dot{H} ^1}  + 2 \int _{t_1} ^{t_2} \Vert \nabla u(\tau) \Vert ^2  _{\dot{H}^1} \, d\tau = \Vert u (t_1) \Vert ^2 _{\dot{H} ^1} + 2  \int _{t_1} ^{t_2} \langle  u(\tau) , |u(\tau)|^2 u(\tau) \rangle _{\dot{H}^1} \, d \tau,
\]
where $t_2 > t_1$.  Note that as $D^2 u \in  L^2([0,\infty); L^2(\RR^4))$, it holds $\nabla u \in  L^2([0,\infty); \dot{H}^1(\RR^4))$,  thus both sides in this equality are finite.

We recall now a generalized Leibniz rule which we will use repeatedly.

 \begin{Theorem}[Theorem 1, Grafakos and Oh \cite{MR3200091}]
\label{leibniz-rule}
Let $f,g \in \mathcal{S} (\RR^n)$,  $\frac{1}{2} < r < \infty$,  $s > \max \{0, \frac{n}{r} - n\}$ and $\frac{1}{r} = \frac{1}{p_1} +  \frac{1}{q_1} =  \frac{1}{p_2} +  \frac{1}{q_2}$,  with $1 < p_1,  p_2, q_1, q_2\leq \infty$.   Then

\begin{displaymath}
\Vert \Lambda^s (fg) \Vert _{L^r} \leq C \left[ \Vert \Lambda^s f \Vert _{L^{p_1}}  \Vert g \Vert _{L^{q_1}} +  \Vert f \Vert _{L^{p_2}}  \Vert \Lambda^s g \Vert _{L^{q_2}} \right].
\end{displaymath}
\end{Theorem}

We have that
\[
\langle \Lambda u\,,\, \Lambda(|u|^2 u) \rangle \leq \| \Lambda u \|_{L^4} \| \Lambda(|u|^2 u) \|_{L^{\frac{4}{3}}} \leq C \| \Lambda u \|_{\dot{H}^1} \| \Lambda(|u|^2 u) \|_{L^{\frac{4}{3}}}\,,
\]
and by Theorem \ref{leibniz-rule}
\begin{align*}
 \| \Lambda(|u|^2 u) \|_{L^{\frac{4}{3}}} & \leq C \left[  \| \Lambda(|u|^2) \|_{L^2}  \| u \|_{L^{4}}   + \| |u|^2 \|_{L^2}  \| \Lambda u \|_{L^{4}} \right] \nonumber \\
 & \leq C \left[  \| \Lambda u \|_{L^4}  \| u \|_{L^{4}}  \| u \|_{L^{4}}   + \| u \|_{L^4}^2  \| \Lambda u \|_{L^{4}} \right]  \\
& = C \| \Lambda u \|_{L^4} \| u \|_{L^4}^2 \nonumber \\
& \leq C \| \Lambda u \|_{\dot{H}^1}  \| u \|_{\dot{H}^1}^2 .  \notag
\end{align*}
As a result of this
\begin{equation}
\label{eqn:Leibniz-rule-term}
\langle u , |u|^2 u \rangle _{\dot{H}^1} = \langle \Lambda u\,,\, \Lambda(|u|^2 u) \rangle \leq  C \| \Lambda u \|_{\dot{H}^1}^2  \| u \|_{\dot{H}^1}^2 ,
\end{equation}
and therefore
\begin{displaymath}%\label{eqn:energy-identity}
\Vert u(t_2) \Vert ^2 _{\dot{H} ^1}  + 2 \left( 1- C  \, \sup_{t > 0} \Vert u(t) \Vert^2 _{\dot{H}^1} \right) \int _{t_1} ^{t_2} \Vert \nabla u(\tau) \Vert ^2  _{\dot{H}^1} \, d\tau \leq \Vert u(t_1) \Vert ^2 _{\dot{H} ^1} \,.
\end{displaymath}
Choosing $\Vert u_0 \Vert  _{\dot{H} ^1}$ small enough,  from \eqref{eq:condition-existence} and Theorem \ref{gustafson-roxanas}  we have that $\| u \|_{L_t^{\infty} \dot{H}_x^1}$ is small for all $t>0$,  hence $ 1 \geq C   \sup_{t > 0} \Vert u(t) \Vert^2 _{\dot{H}^1}$. Then for any pair $t_1 < t_2$ we have
\begin{displaymath}%\label{eqn:energy-identity}
\Vert u(t_2) \Vert ^2 _{\dot{H} ^1} \leq \Vert u(t_1) \Vert ^2 _{\dot{H} ^1} \,,
\end{displaymath}
which entails the result.
\end{proof}

\textbf{Proof of Theorem \ref{main-theorem}:} From Proposition \ref{lyapunov-function}, we know that the $\dot{H} ^1$ norm is a nonincreasing function, hence has a derivative a.e. Then

\begin{align}
\label{eq:preparation-fourier-splitting}
\notag \frac{d}{dt} \Vert u(t) \Vert^2 _{\dot{H} ^1} & = 2 \langle \Lambda  u(t), \partial_t \Lambda u(t) \rangle \\ \notag & = 2 \langle \Lambda  u (t) ,\Lambda  \left( \Delta u (t)  +  |u|^2 u (t) \right)   \rangle \\ \notag &
 = -2 \Vert \nabla u (t) \Vert^2 _{\dot{H} ^1} + 2 \left\langle \Lambda  u (t) , \Lambda \left(  |u|^2 u (t)  \right)  \right\rangle \notag \\ & \leq - 2   \left(1 -  C \Vert  u(t) \Vert^2 _{\dot{H} ^1} \right) \Vert \nabla u(t) \Vert^2 _{\dot{H} ^1} \notag \\ & \leq - \widetilde{C} \Vert \nabla u(t) \Vert^2 _{\dot{H} ^1},
\end{align}
where we have used \eqref{eqn:Leibniz-rule-term} and once again \eqref{eq:condition-existence} together with Theorem \ref{gustafson-roxanas} to have that $\| u \|_{L_t^{\infty} \dot{H}_x^1}$ is small for all $t>0$.

Starting from \eqref{eq:preparation-fourier-splitting},  we can now  use the Fourier Splitting Method,  developed by M.E. Schonbek  to study decay of energy for solutions to parabolic conservations laws \cite{MR571048} and to Navier-Stokes equations \cite{MR775190}, \cite{MR837929}. This method is based on the idea that for many dissipative equations,  for large enough times the remaining energy is concentrated at the low frequencies.  We start by considering a ball $B(t)$ around the origin in frequency space with  continuous, time-dependent radius $r(t)$ such that
\begin{displaymath}
B(t) = \left\{\xi \in \RR^4: |\xi| \leq r(t) = \left( \frac{g'(t)}{\widetilde{C} g(t)} \right) ^{\frac{1}{2}}  \right\},
\end{displaymath}
with $g$ an increasing continuous function such that $g(0) = 1$.  Then
\begin{align*}
\notag \frac{d}{dt} \Vert u(t) \Vert^2 _{\dot{H} ^1} & \leq - \widetilde{C} \Vert \nabla u(t) \Vert^2 _{\dot{H} ^1} =  - \widetilde{C} \int _{\RR ^4} |\xi|^2 \,  ||\xi| \widehat{u} (\xi,t)|^2 \, d \xi \notag \\ & = - \widetilde{C} \int _{B(t)} |\xi|^2 \,  ||\xi| \widehat{u} (\xi,t)|^2 d \xi   - \widetilde{C} \int _{B(t) ^{c}} |\xi|^2 \,  ||\xi| \widehat{u} (\xi,t)|^2 \, d \xi \notag \\ & \leq -  \frac{g'(t)}{g(t)}  \int _{B(t) ^{c}}  ||\xi|  \widehat{u} (\xi,t)|^2 \, d \xi \notag \\ & =      \frac{g'(t)}{g(t)}  \int _{B(t)}  ||\xi| \widehat{u} (\xi,t)|^2 \, d \xi -  \frac{g'(t)}{g(t)}  \int _{\RR ^4}  ||\xi|  \widehat{u} (\xi,t)|^2 \, d \xi.
\end{align*}
Now  multiply both sides by $g(t)$ and rewrite in order to obtain the key inequality
\begin{equation}
\label{eqn:key-inequality}
\frac{d}{dt}   \left( g(t)  \Vert u (t) \Vert _{\dot{H} ^1} ^2 \right) \leq g'(t)  \int _{B(t)} ||\xi|  \widehat{u} (\xi, t)| ^2 \, d \xi.
\end{equation}
We now need a pointwise estimate for $ ||\xi|  \widehat{u} (\xi, t)|$ in $B(t)$ in order to go forward.  We first have
\begin{align}
\int _{B(t)} ||\xi|  \widehat{u} (\xi, t)| ^2 \, d \xi & \leq C \int _{B(t)} \left| e^{- t |\xi| ^2} |\xi| \widehat{u_0} (\xi, t) \right|^2 \, d \xi \notag \\ & + C \int _{B(t)} \left( \int _0 ^t  e^{- (t-s) |\xi| ^2} |\xi| \mathcal{F} \left(|u|^2 u \right) (\xi, s) \, ds \right) ^2 \, d \xi. \notag
\end{align}
The first term corresponds to the linear part, i.e.  the heat equation,  so, by Theorem \ref{characterization-decay-l2} this can be estimated as
\begin{equation}
\label{eq:decay-linear-part}
\int _{B(t)} \left| e^{- t |\xi| ^2} |\xi| \widehat{u_0} (\xi, t) \right|^2 \, d \xi \leq C \Vert e^{t \Delta} \Lambda u_0 \Vert^2 _{L^2} \leq C (1 + t) ^{- \left( 2 + q^{\ast} \right)},
\end{equation}
where $q^{\ast} = r^{\ast} \left( \Lambda u_0 \right)$.

 Now,  for the nonlinear term inside the time integral we have, using Theorem \ref{leibniz-rule} for estimating $ \Vert  \Lambda \left( |u|^2 u \right)  \Vert _{L^1}$ and $\Vert  \Lambda \left( |u|^2 \right) \Vert _{L^{\frac{4}{3}}}$, that
\begin{align}
\left| |\xi| \mathcal{F} \left(|u|^2 u \right) \right| & = \left| \mathcal{F} \left( \Lambda \left( |u|^2 u \right) \right)  \right| \leq \Vert \mathcal{F} \left( \Lambda \left( |u|^2 u \right) \right) \Vert _{L^{\infty}} \leq \Vert  \Lambda \left( |u|^2 u \right)  \Vert _{L^1} \notag \\ & \leq C \Vert \Lambda \left( |u|^2 \right) \Vert _{L^{\frac{4}{3}}} \Vert u \Vert _{L^4}  +  C \Vert |u|^2 \Vert _{L^2}  \Vert \Lambda u \Vert _{L^2} \notag \\ & \leq C \Vert \Lambda u \Vert _{L^2} \Vert u \Vert _{L^4} ^2 + C \Vert u \Vert _{L^4} ^2  \Vert u \Vert _{\dot{H} ^1} \leq C \Vert u \Vert _{\dot{H}^1} ^3. \notag
\end{align}
As a result of this
\begin{align}
\int _{B(t)} \left( \int _0 ^t  e^{- (t-s) |\xi| ^2} |\xi| \mathcal{F} \left(|u|^2 u \right) (\xi, s) \, ds \right) ^2 \, d \xi & \leq C \, {\text{Vol}} \, B(t)\, \,  \left( \int_0 ^t \Vert u(s)  \Vert _{\dot{H}^1} ^3 \, ds \right) ^2 \notag \\ & = C r^4 (t)  \left( \int_0 ^t \Vert u(s)  \Vert _{\dot{H}^1} ^3 \, ds \right) ^2,  \notag
\end{align}
which leads to
\begin{align}
\label{eqn:estimate}
\frac{d}{dt}   \left( g(t)  \Vert u (t) \Vert _{\dot{H} ^1} ^2 \right) & \leq C g'(t)  (1 + t) ^{- \left( 2 + q^{\ast} \right)} \notag \\ & + C g'(t) r^4 (t)  \left( \int_0 ^t \Vert u(s)  \Vert _{\dot{H}^1} ^3 \, ds \right) ^2.
\end{align}
We now use a bootstrap argument to prove the estimate in Theorem \ref{main-theorem}.  More precisely,  we first obtain a weaker estimate which we then plug in in the right hand side of \eqref{eqn:estimate} to find a better one. To do so, we first take $g(t) = \left[\ln (e+t) \right]^3$, which leads to
\begin{displaymath}
r(t) = \left(\frac{g'(t)}{\widetilde{C}g(t)}\right)^{\frac{1}{2}} = \left(  \frac{3}{\widetilde{C}(e+t) \ln (e+t)} \right) ^{\frac{1}{2}}.
\end{displaymath}
From \eqref{eqn:small-solution} we know that $\Vert u(t) \Vert _{\dot{H} ^1} \leq C$ for all $ t >0$, hence from \eqref{eqn:estimate} we obtain
\[
\frac{d}{dt}   \left( \left[\ln (e+t) \right]^3  \Vert u (t) \Vert _{\dot{H} ^1} ^2 \right)  \leq C \frac{\left[\ln (e+t) \right]^2}{e+t}  (1 + t) ^{- \left( 2 + q^{\ast} \right)} + C \frac{1}{(e+t)^3}  t ^2.
\]
Now notice that
\begin{align*}
\int _0 ^t \frac{\left[\ln (e+s) \right]^2}{e+s}  (1 + s) ^{- \left( 2 + q^{\ast} \right)} \, ds & \leq C \int _1 ^{\ln (e+s)} u^2 \, e^{-(2+ q^{\ast})u} \, du  \\ & \leq C \frac{2 \, e^{-(2+ q^{\ast})}}{ \left( 2 + q^{\ast} \right)}   \left(\frac{1}{2} + \frac{1}{2 + q^{\ast}} + \frac{1}{\left( 2 + q^{\ast} \right) ^2}  \right),
\end{align*}

which leads to
\begin{equation}
\label{eqn:first-decay}
\Vert u(t) \Vert _{\dot{H} ^1} ^2  \leq C [\ln (e+t)]^{-2}.
\end{equation}
We now use this estimate to bootstrap. Take $g(t) = (1 + t) ^{\alpha}$, with $\alpha > \max \{ 2 + q^{\ast}, 1 \}$, so $ r(t) = \bigl(\frac{\alpha}{\widetilde{C}(1+t)}\bigr)^{\frac{1}{2}}$.  Note that $\alpha > 0$.  Then,  from \eqref{eqn:estimate} and \eqref{eqn:first-decay} we obtain,  after integrating,  that
\begin{displaymath}
 \Vert u (t) \Vert _{\dot{H} ^1} ^2  \leq C (1+t) ^{-\alpha} + C (1 + t) ^{- \left( 2 + q^{\ast} \right)} + C (1+t)^{-1}  \int _0 ^t \frac{ \Vert u (s) \Vert _{\dot{H} ^1} ^2}{[\ln (e+s)]^4} \, ds,
\end{displaymath}
which leads to
\begin{equation}
\label{eqn:previous-estimate}
 \Vert u (t) \Vert _{\dot{H} ^1} ^2  \leq  C (1 + t) ^{- \left( 2 + q^{\ast} \right)} + C(1+t)^{-1}  \int _0 ^t \frac{ \Vert u (s) \Vert _{\dot{H} ^1} ^2}{ [\ln (e+s)]^4} \, ds.
\end{equation}

We now recall the following Gronwall-type estimate.
\begin{Proposition}[Theorem 1, page 356, Mitrinovi\'{c}, Pe\v{c}ari\'{c} and Fink \cite{MR1190927}] \label{gronwall} Let $x,k:J \to \RR$ continuous and $a,b: J \to \RR$ Riemann integrable in $J = [\alpha, \beta]$. Suppose that $b, k \geq 0$ in $J$. Then, if
\begin{displaymath}
x(t) \leq a(t) + b(t) \int _{\alpha} ^t k(s) x(s) \, ds, \quad t \in J
\end{displaymath}
then
\begin{displaymath}
x(t) \leq a(t) + b(t) \int _{\alpha} ^t a(s) k(s) \exp \left( \int_s ^t b(r)k(r)  \, dr \right) \, ds, \quad t \in J.
\end{displaymath}
\end{Proposition}
Suppose $q^{\ast} >-1$.  Now consider in \eqref{eqn:previous-estimate}
\begin{align}
x(t) & =  \Vert u (t) \Vert _{\dot{H} ^1} ^2,  \quad a(t) = C (1 + t) ^{- \left( 2 + q^{\ast} \right)} \notag \\ b(t) & = C(1+t)^{-1},  \quad k(t) = \frac{1}{[\ln (e+t)]^4}.  \notag
\end{align}
Note that
\begin{align}
\label{eqn:integral-kb}
\int_s ^t b(r) \, k(r) \, dr & = C\int _s ^t \frac{1}{(1 + r) [\ln (e+r)]^4} \, dr  \leq C,  \notag \\ \int_0 ^t a(s) \, k(s) \, ds & =C \int _0 ^t \frac{1}{(1 + s) ^{2 + q^{\ast}} [\ln (e+s)]^4} \, ds  \leq C,
\end{align}
hence using  Proposition \ref{gronwall} in \eqref{eqn:previous-estimate}  we obtain
\begin{equation*}
%\label{eqn:before-bootstrap}
\Vert u (t) \Vert _{\dot{H} ^1} ^2  \leq  C (1 + t) ^{- \left( 2 + q^{\ast} \right)} + C (1+t)^{-1} \leq  C (1+t)^{-1}.
\end{equation*}

Now we consider $q^{\ast} \leq -1$. We rewrite (\ref{eqn:previous-estimate}) as

\begin{displaymath}
(1 + t)  \Vert u (t) \Vert _{\dot{H} ^1} ^2  \leq  C (1 + t) ^{1 - \left( 2 + q^{\ast} \right)} + C  \int _0 ^t \frac{ (1 + s) \Vert u (s) \Vert _{\dot{H} ^1} ^2}{(1 + s)  [\ln (e+s)]^4} \, ds.
\end{displaymath}
We recall  a different  Gronwall-type inequality.

\begin{Proposition}[Corollary 1.2,  page 4,  Bainov and Simeonov \cite{MR1171448}] \label{gronwall-2} Let $a,b,\psi,:J \to \RR$ continuous in $J = [\alpha, \beta]$ and $b \geq 0$.  If $a(t)$ is nondecreasing then
\begin{displaymath}
\psi(t) \leq a(t) +  \int _{\alpha} ^t b(s) \psi(s) \, ds, \quad t \in J
\end{displaymath}
implies
\begin{displaymath}
\psi(t) \leq a(t)  \exp \left( \int_{\alpha} ^t b(s)  \, ds \right) \quad t \in J.
\end{displaymath}

\end{Proposition}
Let
\begin{displaymath}
\psi (t) = (1 + t)  \Vert u (t) \Vert _{\dot{H} ^1} ^2,  \quad a(t) = C (1 + t) ^{1 - \left( 2 + q^{\ast} \right)}, 
\end{displaymath}
\[ \, b(t) = \frac{ C}{(1 + t)  [\ln (e+t)]^4}. \]
Notice that $q^{\ast} \leq -1$ implies that $a(t)$ is nondecreasing,  then Proposition \ref{gronwall-2}  and (\ref{eqn:integral-kb}) lead to
\begin{displaymath}
(1 + t)  \Vert u (t) \Vert _{\dot{H} ^1} ^2  \leq  C (1 + t) ^{1 - \left( 2 + q^{\ast} \right)}.
\end{displaymath}
We have thus proved Theorem \ref{main-theorem}.  $\Box$

\end{document}